\documentclass[11pt]{amsart} 
\usepackage{amssymb}
\usepackage{amsmath}
\usepackage{tikz} 

\theoremstyle{plain}
\newtheorem{theorem}{Theorem}
\newtheorem{lemma}[theorem]{Lemma}
\newtheorem{pro}[theorem]{Proposition}

\newtheorem{conj}[theorem]{Conjecture}

\theoremstyle{definition}

\begin{document}
\title{On Gardner's conjecture}

\author{G\'abor Kun}

\address{Alfr\'ed R\'enyi Institute of Mathematics, H-1053 Budapest, Re\'altanoda u. 13-15., Hungary}
\address{Institute of Mathematics, E\"otv\"os L\'or\'and University, P\'azm\'any P\'eter s\'et\'any 1/c, H-1117 Budapest, Hungary}

\email{kungabor@renyi.hu}

\thanks{The author's work on the project leading to this application has received funding from the European Research Council (ERC) under the European Union's Horizon 2020 research and innovation programme (grant agreement No. 741420), from the \'UNKP-20-5 New National Excellence Program of the Ministry of Innovation and Technology from the source of the National Research, Development and Innovation Fund and from the J\'anos Bolyai Scholarship of the Hungarian Academy of Sciences.}

\maketitle

\begin{abstract}
Gardner conjectured that if two bounded measurable sets $A,B \subset \mathbb{R}^n$ are equidecomposable by a set of isometries $\Gamma$ generating an amenable group then $A$ and $B$ admit a measurable equidecomposition by all isometries. Cie\'sla and Sabok asked if there is a measurable equidecomposition using isometries only in the group generated by $\Gamma$. We answer this question negatively. 
\end{abstract}

\section{Introduction}

Let $\Gamma$ be a set of isometries  of $\mathbb{R}^n$. We say that the sets $A, B \subseteq \mathbb{R}^n$ are $\Gamma$-{\it equidecomposable} if there are finite partitions $A=\cup_{n=1}^k A_n, B=\cup_{n=1}^k B_n$ and isometries $\gamma_1, \dots ,\gamma_k \in \Gamma$ such that $A_i=\gamma_iB_i$ holds for every 
$1 \leq i \leq k$. $A$ and $B$ admit a $\Gamma$-equidecomposition if and only if there exists a finite set of generators in $\Gamma$ such that the bipartite restriction 
of the Schreier graph of the generated subgroup $\langle \Gamma \rangle$ with respect to the classes $A$ and $B$ admits a  perfect matching. The Hall condition is necessary for the existence of an equidecomposition, and if $\Gamma$ is finite then it is also sufficient. An equidecomposition is {\it measurable} if the sets in the partitions are Lebesgue measurable. We recommend M\'ath\'e \cite{M} on measurable equidecompositions, 
Kechris and Marks \cite{KM} on measurable matchings and Tomkowicz and Wagon \cite{TW} on the Banach-Tarski paradox.
The Banach-Tarski paradox \cite{BT} claims that any two bounded subsets of $\mathbb{R}^3$ with nonempty interior are equidecomposable. 
A group is {\it amenable} if it admits no paradoxical decomposition. Von Neumann introduced amenable groups in order to explain the paradox \cite{vN}. 
Mycielski showed that if two measurable sets are equidecomposable by an amenable group then they have equal Lebesgue measure \cite{My}. 
This is the largest class of groups where the Hall condition may imply the existence of a measurable equidecomposition. 
Gardner made the following conjecture. 

\begin{conj} \cite{G}
Consider the bounded measurable sets $A, B \subseteq \mathbb{R}^n$ and the set of isometries $\Gamma$. Assume that $\Gamma$ generates an amenable group. If $A$ and $B$ are $\Gamma$-equidecomposable then they admit a measurable equidecomposition.
\end{conj}

The most famous open question on equidecompositions was Tarski's circle squaring problem \cite{T} asking if the disc and the square of unit area admit an equidecomposition. Laczkovich solved this positively \cite{Lsquaring}.
His construction was not measurable, and left the measurable circle squaring problem open for decades.
Grabowski, M\'ath\'e and Pikhurko have found a measurable equidecomposition \cite{GMP}, while Marks and Unger have given a Borel equidecomposition \cite{MU}. 
Both of these works rely heavily on the details of Laczkovich's proof. Gardner's conjecture would provide a simple solution using only the existence of an equidecomposition.

Note that Gardner's conjecture does not require that the isometries in the measurable equidecomposition are in $\Gamma$. Laczkovich has found an example of two measurable sets admitting a $\Gamma$-equidecomposition but no measurable $\Gamma$-equidecomposition \cite{Lmatching}.
Cie\'sla and Sabok have studied the conjecture for probability measure preserving group actions.
They asked if the elements of $\langle \Gamma \rangle$ are sufficient (Question 4, \cite{CS}), and proved this for abelian groups if the sets are uniformly distributed.
We give a negative answer to their question.

\begin{theorem}~\label{only}
Consider an irrational number $0<\alpha<1$ and the following set consisting of four isometries of $\mathbb{R}$: \\
$$\begin{displaystyle} \Gamma = 
\{ x \mapsto x, \text{ } x \mapsto x+ 2\alpha,\text{ } x \mapsto 2-x,\text{ } x \mapsto 2\alpha-x \}. \end{displaystyle}$$
The intervals $[0,1]$ and $[\alpha, 1+\alpha]$ admit a $\Gamma$-equidecomposition, but no measurable 
$\langle \Gamma \rangle$-equidecomposition.
\end{theorem}

Technically, our example is not on a probability measure space, but on the real line. However, it can be easily modified to one on the closed cycle.
Note that there is an equidecomposition using one single translation by $\alpha$. The restriction of the bipartite Schreier graph of these four isometries in $\Gamma$ is essentially Laczkovich's construction in \cite{Lmatching}. 

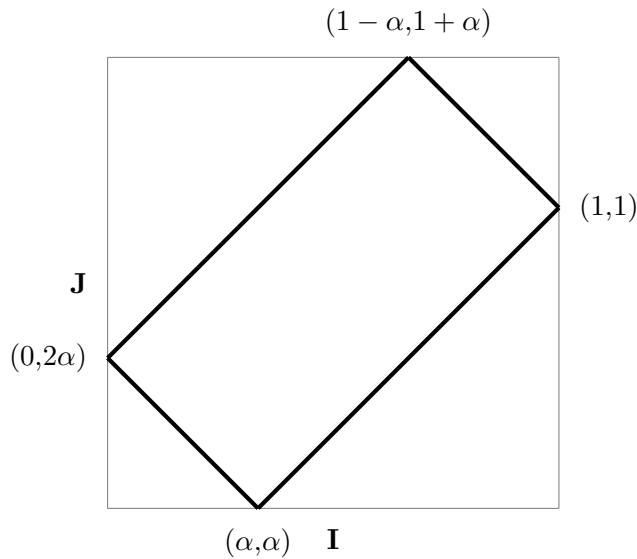
\begin{center}

\begin{tikzpicture}

\node at (2,2) [label=left:$(0{,}2\alpha)$] {};
\node at (4,0) [label=below:$(\alpha{,}\alpha)$] {};
\node at (8,4) [label=right:$(1{,}1)$] {};
\node at (6,6) [label=above:$(1-\alpha{,}1+\alpha)$] {};
\node at (2,3) [label=left:${\bf J}$] {};
\node at (5,0) [label=below:${\bf I}$] {};

\draw[gray, thin] (2,0) -- (2,6);
\draw[gray, thin] (2,6) -- (8,6);
\draw[gray, thin] (8,6) -- (8,0);
\draw[gray, thin] (8,0) -- (2,0);
\draw[black, ultra thick] (2,2) -- (6,6);
\draw[black, ultra thick] (6,6) -- (8,4);
\draw[black, ultra thick] (8,4) -- (4,0);
\draw[black, ultra thick] (4,0) -- (2,2);

\end{tikzpicture}

\end{center}

\section{The proof}

Consider the Borel bipartite graph $G$ whose classes $I$ and $J$ are Borel isomorphic to the intervals $[0,1]$ and $[\alpha, 1+\alpha]$, respectively, and the set of edges corresponds to the union of the restriction of the graph of the four isometries. See $E(G) \subseteq I \times J$ in the figure.
Laczkovich showed that this graph admits a perfect matching, but no measurable perfect matching. Hence $[0,1]$ and $[\alpha, 1+\alpha]$  admit an equidecomposition, but no measurable equidecomposition using these four isometries.

Four vertices of $G$ (in different components) have degree one,
the other vertices have degree two. Laczkovich proved that none of the connected components is an even path, hence every connected component of $G$ is a bi-infinite path, a semi-infinite path, an even cycle or an odd path.
We will prove the theorem by contradiction showing that if $[0,1]$ and $[\alpha, 1+\alpha]$ admitted a measurable $\Gamma$-equidecomposition then $G$ would admit a measurable perfect matching. The next lemma shows that elements of $\langle \Gamma \rangle$ do not move vertices of $G$ too far.

\begin{lemma}~\label{lem}
For every $\gamma \in \langle \Gamma \rangle$ and $y \in [0,1]$ such that $\gamma(y) \in [0,1]$
the distance of the corresponding vertices in $I$ is at most $2|b|$ in $G$, where $\gamma(x) = ax + 2\alpha b + 2c$.
\end{lemma}

\begin{proof}
We prove by induction on $|b|$. If $|b|=0$, that is, $\gamma(x)=ax+2c$ then we can easily check those
isometries $\gamma$ such that $\gamma([0,1]) \cap [0,1] \neq \emptyset$.
If $a=1$ then $c=0$ and $\gamma$ is the identity. If $a=-1$ then either $c=0, \gamma(x)=-x$ and $y=\gamma(y)=0$, or $c=1, \gamma(x)=-x+2$  and $y=\gamma(y)=1$.

Now consider $\gamma(x) = ax + 2\alpha b + 2c $, where $|b|>0$, and $y \in I$. If $b<0$ and $\gamma(y) \leq 1-2\alpha$ then by induction
$z=\gamma(y)+2\alpha =  ax + 2\alpha(b+1) + 2c \in I$ satisfies $\mathrm{dist}(y,z) \leq 2|b+1| = 2|b|-2$. Since $\mathrm{dist}(\gamma(y),z) \leq 2$ 
we get $\mathrm{dist}(y,\gamma(y)) \leq 2|b|$. 
If $\gamma(y) > 1-2\alpha$ then consider $z=2-(\gamma(y)+2\alpha) =  -ay - 2\alpha(b+1) + (2-2c) \in I$. Since $|b+1| = |b|-1$ we get the same inequalities as in the previous case.
If $b>0$ and $\gamma(y) > 2\alpha$ then we can use $z=\gamma(y)-2\alpha =  ay + 2\alpha (b-1) + 2c$, since $|b-1|=|b|-1$.  
If $b>0$ and $\gamma(y) \leq 2\alpha$ then consider $z=(2\alpha-\gamma(y))=-ay + 2\alpha (1-b) - 2c$ as an intermediate element using that $|1-b|=|b|-1$.
\end{proof}

We will study the union of bi-infinite paths of $G$. Recall that a {\it ray} in a graph is an infinite sequence of vertices in which each vertex 
appears at most once and each two consecutive vertices are adjacent.
Given a graph $H$ and a natural number $K$ denote by $H^K$ the graph with vertex set $V(H^K)=V(H)$, two vertices are adjacent in $H^K$ 
if their distance is odd and at most $K$.

\begin{pro}~\label{prop}
Let  $K$ be an odd natural number and $H$ be an acyclic $2$-regular measure preserving bipartite Borel graph over a standard Borel probability measure space. If $H^K$ admits a measurable perfect matching then $H$ admits a measurable perfect matching, too.
\end{pro}

\begin{proof}
We denote the probability measure by $\lambda$.
Let $V(H)=V(H^K)=A \cup B$ be the Borel partition of these bipartite graphs into two independent sets.
Assume that $H^K$ admits a measurable perfect matching $M$. Given a vertex $x$ let $M(x)$ be its pair
and $r(x)$ the ray starting at $x$ and containing $M(x)$. Given $x,y \in A$ we write $\varphi(x,y)$ if
$dist_H(x,y)=2,  x \in r(y)$ and $y \in r(x)$. Note that the formula $\varphi$ defines a symmetric relation, and for every $x$ there exists at most one $y$
such that $\varphi(x,y)$ holds.
Consider the set of vertices $S(M)=\{ a \in A: \exists a' \in A \text{ }\varphi(a,a') \}$. \\
We will change $M$ and match every $a\in S(M)$ to $M(a')$ and $a'$ to $M(A)$, where $a'$ is chosen such that $\varphi(a,a')$ holds.

{\bf Claim 1:} Let $M$ be a measurable perfect matching of $H^K$ and consider \\ 
$$\begin{displaystyle} M'=\{(s,t): \varphi(s,M(t)) \text{ holds}\} \cup M \setminus \{(s,M(s)): s \in S(M) \}. \end{displaystyle}$$ \\
The set $M'$ is a measurable perfect matching of $H^K$ and \\
$$\begin{displaystyle} \int_A \mathrm{dist}_H(x,M'(x)) d \lambda(x) \leq \int_A \mathrm{dist}_H(x,M(x)) d \lambda(x) - \lambda(S(M)).\end{displaystyle}$$ 

\begin{proof}
The matching $M'$ covers the same set of vertices as $M$, since the added and removed pairs of edges are in one-to-one correspondence. 
Note that $\mathrm{dis}t_H(a,M'(a)) \leq \mathrm{dist}_H(a',M(a'))$, and $\mathrm{dist}_H(a',M'(a')) \leq \mathrm{dist}_H(a,M(a))$, hence $(a,M(a)), (a'M(a')) \in E(H^K)$.
The inequality \\ $\mathrm{dist}_H(a,M'(a))+\mathrm{dist}_H(a',M'(a'))\leq \mathrm{dist}_H(a',M(a'))+\mathrm{dist}_H(a,M(a))-2$ also holds proving the claim.
\end{proof}

We will find a measurable almost perfect matching $M$ of $H^K$ such that $\lambda(S(M))=0$.
We start with a measurable perfect matching $M_1$ and define a sequence of measurable perfect matchings.
If $\lambda(S(M_n))>0$ then we apply Claim 1 to $M_n$ in order to get $M_{n+1}$.
Note that the measure of the set of vertices where the matching is changed is at most $2\lambda(S(M_n))$.
The inequality 

$$\begin{displaystyle}\int_A \mathrm{dist}_H(x,M_n(x)) d \lambda(x) \geq  \lambda(S(M_n)) + \int_A \mathrm{dist}_H(x,M_{n+1}(x)) d \lambda(x) \end{displaystyle}$$ 

\noindent
implies $\sum_{n=1}^{\infty} \lambda(S(M_n)) \leq  \int_A \mathrm{dist}_H(x,M_1(x)) \leq K$. This sequence converges to a measurable almost perfect matching $M$ by the Borel-Cantelli lemma. Obviously $\lambda(S(M))=0$.

{\bf Claim 2:}
For almost every $a \in A$ for every $c \in A$ in the component of $a$ one of the two rays $r(a)$ and $r(c)$ contains the other.

\begin{proof} 
Let $a,c \in A$ be in the same component. It suffices to show that if none of the two rays $r(a)$ and $r(c)$ contains the other then the component contains either a vertex in $S(M)$ or an unmatched vertex.

If $a \in r(c)$ and $c \in r(a)$ then there is a vertex of $S(M)$ between $a$ and $c$: the vertex $d$ closest to $a$ such that $c \notin r(d)$ is in $S(M)$.
If there is no such vertex $d$ then $c \in S(M)$. The proof of the other case will be based on this observation, too.

Now assume that  $a \notin r(c)$ and $c \notin r(a)$. First, we find such a pair in the same component that their distance is minimal.
If the distance of $a$ and $c$ is greater than two then consider a vertex $d \in A$ between them. The ray $r(d)$ can not contain both $a$ and $c$, and $d$ is not contained
in any of the two rays $r(a)$ and $r(c)$. Hence one of the pairs $(a,d)$ and $(c,d)$ will satisfy the conditions. Therefore,
we may assume that $a$ and $c$ are at distance two. Let $b \in B$ denote the vertex between $a$ and $c$.
If $b$ is matched to a vertex in $r(a)$ then there is a vertex of $S(M)$ between this vertex and $a$. If $b$ is matched to a vertex in $r(c)$ 
then there is a vertex of $S(M)$ between this vertex and $c$.
\end{proof}

We define a measurable perfect matching of $H$ using $M$. Consider the components where for any two vertices $a,c \in A$ one of the two rays $r(a)$ and $r(c)$ contains the other. Let us match every vertex $a \in A$ in such a component to the unique vertex in $r(a)$ adjacent to $a$. This gives a measurable almost perfect matching. It can be extended on the nullset of the other components to a measurable perfect matching. 
\end{proof}

Suppose for a contradiction that $[0,1]$ and $[\alpha, 1+\alpha]$ admit a measurable $\Gamma$-equidecomposition. 
The set of vertices in bi-infinite paths of $G$ is measurable. Lemma~\ref{lem} and Proposition~\ref{prop} imply that this set  admits a measurable 
perfect matching. The other components are even cycles, odd paths or semi-infinite paths, since $G$ admits a perfect matching and its maximum degree is two. Therefore, these admit a measurable perfect matching, too. This contradicts Laczkovich's result and finishes the proof of the Theorem.


\begin{thebibliography}{10}

\bibitem{BT}
Stefan Banach  and Alfred Tarski. "Sur la d\'ecomposition des ensembles de points en parties respectivement congruentes." Fund. math 6.1 (1924): 244-277.

\bibitem{CS}
Tomasz Cie\'sla and Marcin Sabok. "Measurable Hall's theorem for actions of abelian groups." arXiv preprint arXiv:1903.02987 (2019).

\bibitem{G}
R. J. Gardner. Measure theory and some problems in geometry. Atti Sem. Mat. Fis. Univ. Modena, 39(1), 51-72 (1991).

\bibitem{GMP}
Lukasz Grabowski, Andr\'as M\'ath\'e and Oleg Pikhurko. "Measurable circle squaring." Annals of Mathematics (2017): 671-710.

\bibitem{KM}
Alexander S. Kechris and Andrew S. Marks. "Descriptive graph combinatorics, 2015." Preprint.

\bibitem{Lmatching}
Mikl\'os Laczkovich. "Closed sets without measurable matching." Proceedings of the American Mathematical Society 103.3 (1988): 894-896.

\bibitem{Lsquaring}
Mikl\'os Laczkovich. "Equidecomposability and discrepancy; a solution of Tarski's circle-squaring problem." Journal f\"ur die reine und angewandte Mathematik 1990.404 (1990): 77-117.


\bibitem{MU}
Andrew S. Marks and Spencer T. Unger. "Borel circle squaring." Annals of Mathematics (2017): 581-605.

\bibitem{M}
Andr\'as M\'ath\'e. "Measurable equidecompositions." Proceedings of the International Congress of Mathematicians. Vol. 2. 2018.

\bibitem{My}
Jan Mycielski. "Finitely additive invariant measures. I." In Colloquium Mathematicum, Vol. 42, No. 1, (1979) pp. 309-318. Institute of Mathematics Polish Academy of Sciences.

\bibitem{vN}
John von Neumann. "Zur allgemeinen Theorie des Masses." Fundamenta Mathematicae 13.1 (1929): 73-116.


\bibitem{T}
Alfred Tarski. "Probl\'eme 38", Fundamenta Mathematicae, 7: 381 (1925).

\bibitem{TW}
Grzegorz Tomkowicz and Stan Wagon. The Banach-Tarski Paradox. Vol. 163. Cambridge University Press, 2016.

\end{thebibliography}
\end{document}